\documentclass[a4paper,reqno,12pt]{amsart}
\usepackage[T1]{fontenc}
\usepackage[english]{babel}
\usepackage{mathptmx,amsmath,amssymb,mathtools}
\usepackage{fullpage}
\usepackage{microtype}
\usepackage[hidelinks]{hyperref}

\newtheorem{theorem}{Theorem}[section]
\newtheorem{proposition}[theorem]{Proposition}
\newtheorem{lemma}[theorem]{Lemma}
\newtheorem{corollary}[theorem]{Corollary}
\newtheorem{definition}[theorem]{Definition}

\newcommand{\bC}{\mathbb C}

\newcommand{\PSH}{\operatorname{PSH}}
\newcommand{\cO}{\mathcal O}
\newcommand{\cI}{\mathcal I}
\newcommand{\cJ}{\mathcal J}
\newcommand{\mm}{\mathfrak m}
\newcommand{\ddc}{dd^c}
\newcommand{\reg}{\mathrm{reg}}
\newcommand{\sing}{\mathrm{sing}}
\DeclareMathOperator{\Ass}{Ass}
\DeclareMathOperator{\Ann}{Ann}
\DeclareMathOperator{\ord}{ord}
\DeclareMathOperator{\Proj}{Proj}
\DeclareMathOperator{\Bl}{Bl}

\numberwithin{equation}{section}
\title[Local maximality and analytic singularities]{Local-to-global maximality for plurisubharmonic functions with locally analytic singularities}
\author[Thai Duong Do]{Thai Duong Do$^{1,2}$}
\address{$^{1}$Institute for Artificial Intelligence, VNU University of Engineering and Technology, Vietnam National University, Hanoi, Vietnam.}
\address{$^{2}$Department of Mathematics, National University of Singapore, 10 Lower Kent Ridge Road, 119076, Singapore.}
\email{dtduong@vnu.edu.vn, dtduong@nus.edu.sg}
\author[Pham Hoang Hiep]{Pham Hoang Hiep$^{1}$}
\address{$^{1}$Institute for Artificial Intelligence, VNU University of Engineering and Technology, Vietnam National University, Hanoi, Vietnam.}
\email{phhiepvn@vnu.edu.vn}
\date{}
\begin{document}
\begin{abstract}
We prove that a locally maximal plurisubharmonic function on a domain in $\bC^n$, $n\geq2$, is maximal if it has locally analytic singularities with a two-sided locally bounded remainder. Using McAdam's theorem, we show that local maximality forces the analytic spread of each defining ideal to be less than $n$. We use normalized blow-ups and Siu's analyticity theorem to compare singularities through generic divisorial Lelong numbers. We also obtain a criterion for maximality in terms of maximality outside the pole set and the analytic spread of the local defining ideals.
\end{abstract}
\maketitle
\tableofcontents

\section{Introduction and main results}
A plurisubharmonic function $u$ on a domain $\Omega$ is maximal if every $v\in\PSH(\Omega)$ satisfying $v\leq u$ outside a compact subset of $\Omega$ also satisfies $v\leq u$ on $\Omega$. For locally bounded functions, the Bedford--Taylor theory characterizes maximality by the vanishing of the complex Monge--Amp\`ere measure \cite{BT76,BT82}. Maximality is therefore a local property in that class. For unbounded functions, the local-to-global question requires control of comparison functions near the singular set.

Demailly extended the complex Monge--Amp\`ere operator to plurisubharmonic functions that are locally bounded near the boundary of bounded domains in $\bC^n$ and developed applications to intersection theory \cite{D93}.

Cegrell developed a theory of the complex Monge--Amp\`ere operator for unbounded plurisubharmonic functions \cite{C98,C04}. B\l ocki characterizes its domain of definition under decreasing approximation in dimension two \cite{B04} and in arbitrary dimension \cite{B06}.

On a bounded hyperconvex domain, the negative functions in this domain of definition form Cegrell's class $\mathcal E$. Within the domain of definition, maximality is equivalent to the vanishing of the Monge--Amp\`ere measure and is consequently local, by \cite[Proposition~2.2]{B04}. B\l ocki also shows that, for general maximal functions, the Monge--Amp\`ere measures of the canonical truncations need not tend to zero \cite{B09}. Functions with analytic singularities need not belong to this domain of definition.

In \cite[Theorem~1]{C09}, Cegrell characterized a negative maximal function $u$ on a bounded hyperconvex domain $\Omega$ by the existence of functions $u_j\in\mathcal E_0(\Omega)\cap C(\overline\Omega)$ such that $u_j\geq u$, $u_j\to u$ pointwise, and $(\ddc u_j)^n\to0$ weakly. The sequence is not required to be decreasing. In \cite{C08}, maximal functions serve as generalized boundary values for the Dirichlet problem.

Further comparison principles and Dirichlet problems in Cegrell's classes were studied in \cite{H06,AC07,KH09,ACCP09}. In the present paper, comparison near the singular set is obtained from the defining analytic ideals, without an energy-class assumption.

In this paper we consider functions with locally analytic singularities. More precisely, near each point we assume
\begin{equation}\label{eq:model-intro}
 u=c\log|f|+O(1),\qquad c>0,
\end{equation}
where $f=(f_1,\ldots,f_m)$ is a holomorphic tuple and
$|f|=(\sum_{j=1}^{m}|f_j|^2)^{1/2}$. The remainder $O(1)$ is locally bounded on both sides; no continuity or plurisubharmonicity is assumed. Both $c$ and $f$ may depend on the neighborhood.

\begin{theorem}\label{thm:main}
Let $\Omega\subset\bC^n$, $n\geq2$, be a domain. Suppose that $u\in\PSH(\Omega)$, $u\not\equiv-\infty$, has locally analytic singularities in the sense of \eqref{eq:model-intro}. If $u$ is locally maximal on $\Omega$, then $u$ is maximal on $\Omega$.
\end{theorem}

The singular set
\[
 E=\{u=-\infty\}
\]
is a proper closed analytic subset, and $u$ is locally bounded on $\Omega\setminus E$. It remains to control comparison functions near $E$.

We shall use the following comparison of singularities. Here $\ell(I)$ denotes the analytic spread of an ideal, recalled in Section~2.
\begin{proposition}\label{prop:intro-rigidity}
Let $\Omega\subset\bC^n$, $n\geq2$, be a domain, and let $u\in\PSH(\Omega)$, $u\not\equiv-\infty$. Suppose that $\Omega$ has an open covering $(D_j)$ on which
\[
 u=c_j\log|\cI_j|+O(1),\qquad c_j>0,
 \qquad \ell((\cI_j)_z)<n\quad\text{for every }z\in D_j,
\]
where $\cI_j$ is a nonzero coherent ideal on $D_j$. Set $E=\{u=-\infty\}$. If $v\in\PSH(\Omega)$ satisfies $v\geq u$ and $v=u$ outside a compact subset of $\Omega$, then for every $z\in E$ there are a neighborhood $D_z$ and a finite constant $C_z$ such that
\begin{equation}\label{eq:rigidity}
 0\leq v-u\leq C_z\quad\text{on }D_z\setminus E.
\end{equation}
\end{proposition}

The proof of Theorem~\ref{thm:main} combines a necessary condition on the local ideals with the comparison in Proposition~\ref{prop:intro-rigidity}. Using McAdam's theorem, Lemma~\ref{lem:spread} shows that local maximality implies $\ell(\cI_z)<n$. After taking suitable powers and integral closures, we glue finitely many local ideals and form their normalized blow-up. The spread bound ensures that the image of each prime component of the total transform of $E$ has positive dimension and therefore meets the complement of the compact set. Siu's analyticity theorem \cite{Siu} asserts that $\{x:\nu(T,x)\ge a\}$ is analytic for every $a>0$ when $T$ is a positive closed $(1,1)$-current on a complex manifold. Consequently, equality of the pullbacks on a nonempty open part of each prime divisor determines their generic divisorial Lelong numbers. Siu's decomposition then permits subtraction of local logarithmic potentials, giving \eqref{eq:rigidity}.

Rashkovskii--Sigurdsson \cite[Theorem~4.3]{RS} established local maximality properties of Green functions with singularities along complex spaces. Rashkovskii \cite[Theorem~2 and Corollary~3]{R09} subsequently proved propagation of analytic singularities and global maximality results for these Green functions. His propagation theorem assumes that the ideal is integrally generic, meaning that its analytic spread equals the codimension of its zero set at each point of that set, and propagates an upper singularity bound from an open set meeting every irreducible component of that set. Proposition~\ref{prop:intro-rigidity} uses the condition $\ell(\cI_z)<n$, but imposes the different comparison hypotheses $v\geq u$ and $v=u$ outside a compact set. This equality supplies the divisorial bounds on the normalized blow-up. Together with Lemma~\ref{lem:spread}, the proposition proves local-to-global maximality for functions satisfying \eqref{eq:model-intro}, with only a locally bounded remainder and without a representation as a Green function.

In particular, Theorem~\ref{thm:main} applies when $u=c\log|\cI|+O(1)$ for a fixed nonzero coherent ideal $\cI$ and a constant $c>0$. Thus a locally maximal function $\log|f|$, for any finite holomorphic tuple $f$ not identically zero, is maximal.

The proof also characterizes maximality in terms of local maximality outside $E$ and the analytic spread of the defining ideals at points of $E$.

\begin{corollary}\label{cor:criterion}
Let $\Omega\subset\bC^n$, $n\geq2$, be a domain, and let $u\in\PSH(\Omega)$, $u\not\equiv-\infty$. Suppose that on an open covering $(D_j)$,
\[
 u=c_j\log|\cI_j|+O(1),\qquad c_j>0,
\]
where each $\cI_j$ is a coherent ideal with nonzero stalks. Set $E=\{u=-\infty\}$. The following conditions are equivalent:
\begin{enumerate}
\item[(i)] $u$ is maximal on $\Omega$.
\item[(ii)] $u$ is locally maximal on $\Omega$.
\item[(iii)] $u$ is locally maximal on $\Omega\setminus E$, and
\[
 \ell((\cI_j)_z)<n\qquad\text{for every }j\text{ and every }z\in D_j\cap E.
\]
\end{enumerate}
In \textup{(iii)}, local maximality on $\Omega\setminus E$ can equivalently be replaced by
\[
 (\ddc u)^n=0\quad\text{on }\Omega\setminus E,
\]
where the measure is understood in the Bedford--Taylor sense.
\end{corollary}

\section{Preliminaries and auxiliary results}
We use $d^c=\frac{i}{2\pi}(\bar\partial-\partial)$, so that
$\ddc=\frac{i}{\pi}\partial\bar\partial$ and
$\ddc\log|f|=[\operatorname{div}(f)]$ for a scalar-valued holomorphic function $f$ not identically zero.
All plurisubharmonic functions take values in $[-\infty,+\infty)$.

\begin{definition}
A function $u\in\PSH(D)$ is maximal on $D$ if, for every $v\in\PSH(D)$ such that $\{v>u\}\Subset D$, one has $v\leq u$. Here $A\Subset D$ means that $\overline A$ is compact and contained in $D$. The function $u$ is locally maximal if every point has an open neighborhood on which $u$ is maximal.
\end{definition}
We use this definition throughout. Maximality passes to open subsets: a competitor on an open subset can be replaced by its maximum with $u$ and glued to $u$ outside its compact support. More explicitly, if $G\subset D$ is open and $v\in\PSH(G)$ satisfies $\{v>u\}\Subset G$, then $\max\{u,v\}$ agrees with $u$ near the boundary of $G$ in $D$. Defining it to be $u$ on $D\setminus G$ therefore gives a plurisubharmonic function on $D$. Maximality on $D$ forces this extension to equal $u$, which proves maximality on $G$.

\begin{definition}
A function $u\in\PSH(D)$ has locally analytic singularities if each point has a neighborhood on which
\begin{equation}\label{eq:model}
 c\log|f|-M\leq u\leq c\log|f|+M
\end{equation}
for a finite holomorphic tuple $f$ not identically zero, a constant $c>0$, and a finite constant $M$. The inequalities include points of the common zero set of $f$.
\end{definition}
At a point where $u$ is locally bounded, the unit ideal gives such a presentation. On each neighborhood where \eqref{eq:model} holds, $E=\{f=0\}$; this proves the assertions about $E$ made in the introduction.

The two inequalities in \eqref{eq:model} will be used separately. The lower bound prevents the singularity of $u$ from being stronger than its ideal model, while the upper bound prevents it from being weaker. After principalizing the ideal, these bounds give a bounded residual function. Its plurisubharmonicity will follow from subtraction of the divisorial current, rather than from any regularity assumption on the original remainder.

For a nonzero ideal $I=(f_1,\ldots,f_m)$ in the local analytic ring $\cO_{\Omega,z}$, write $|I|=(\sum_{j=1}^{m}|f_j|^2)^{1/2}$. Different finite generating systems have comparable norms after shrinking. Thus $\log|I|$ is well defined up to a locally bounded term. The integral closure $\overline I$ satisfies
\begin{equation}\label{eq:closure}
 g\in\overline I\quad\Longleftrightarrow\quad |g|\leq C|I|\text{ near }z
\end{equation}
for some finite $C$; see \cite[Theorem~2.1, equivalence of (i) and (vi)]{LT}. The coherence of the integral closure of a coherent ideal is proved in \cite[Theorem~2.9]{LT}.

For a Noetherian local ring $(R,\mm)$, let
\[
 \mathcal R(I)=\bigoplus_{k=0}^{\infty}I^kt^k,
 \qquad \mathcal F(I)=\bigoplus_{k=0}^{\infty}I^k/\mm I^k,
 \qquad \ell(I)=\dim\mathcal F(I).
\]
These are the Rees algebra, the fiber cone, and the analytic spread. For a proper ideal, one has $\ell(I)\leq\dim R$; for the unit ideal, $\mathcal F(R)= (R/\mm)[t]$ and $\ell(R)=1$. For a nonzero ideal in a local analytic ring and $N\geq1$,
\begin{equation}\label{eq:spread-invariance}
 \ell(\overline{I^N})=\ell(I^N)=\ell(I).
\end{equation}
Indeed, $\mathcal F(I^N)$ is the $N$th Veronese subalgebra of $\mathcal F(I)$, so the two algebras have the same dimension. For $J=\overline I$, the inclusion $R[It]\subset R[Jt]$ is finite and integral. The induced map of spectra is finite and surjective, also over the closed point of $\operatorname{Spec}R$; hence the fiber cones have the same dimension. Applying this to $I^N$ proves \eqref{eq:spread-invariance}.

In this argument it is the geometry of the closed fibers that matters.
Writing $\kappa=R/\mm$, these fibers are the spectra of
$R[It]\otimes_R\kappa=\mathcal F(I)$ and
$R[Jt]\otimes_R\kappa=\mathcal F(J)$. Base change gives a finite
surjective morphism
\[
 \operatorname{Spec}\mathcal F(J)\longrightarrow
 \operatorname{Spec}\mathcal F(I),
\]
so their dimensions agree. One does not need to regard the map of fiber
cones as an inclusion. This observation explains why replacing an ideal
by its integral closure preserves the fiber-dimension information used
later, as well as its logarithmic singularity type.

We use the following consequence of McAdam's theorem \cite[Theorem~3]{McA}: for the excellent regular local domain $R=\cO_{\Omega,z}$ and a nonzero proper ideal $I$, the equality $\ell(I)=\dim R$ implies
\begin{equation}\label{eq:mcadam}
 \mm\in\Ass(R/\overline{I^k})\quad\text{for some }k\geq1.
\end{equation}
This regular local domain satisfies the altitude formula required in McAdam's theorem.

For a nonzero coherent ideal $\cJ$ on a complex manifold $D$, its normalized blow-up $\pi:Y\to D$ is proper and surjective. The space $Y$ is normal, the map is an isomorphism off $V(\cJ)$, and
\[
 \cJ\cO_Y=\cO_Y(-B)
\]
for an effective Cartier divisor $B$; see \cite[Sections~3 and 4.1]{LT}. By Oka's normalization theorem \cite{Oka51}, the normalization of a reduced complex space is finite, proper, and biholomorphic over the normal locus. In particular, normalization preserves the dimension of each fiber of the blow-up.

On a complex manifold, Siu's decomposition expresses a positive closed $(1,1)$-current as its divisorial part plus a positive residual current \cite{Siu}. We denote the generic Lelong number along a prime divisor $H$ by $\nu_H(T)$. This number is determined on any nonempty relatively open subset of $H$. Indeed, the sets on which the Lelong number is larger than its generic value form a countable union of proper analytic subsets. Locally bounded perturbations of plurisubharmonic functions preserve Lelong numbers.

\begin{samepage}
\begin{lemma}[{\cite[Proposition~2.3]{RS}}]\label{lem:maxpush}
Let $f:X\to Y$ be a proper surjective holomorphic map between
complex manifolds of the same dimension. If $u\in\PSH(X)$, then
\[
 f_{\max}u(z):=\max\{u(\zeta):\zeta\in f^{-1}(z)\}
\]
is plurisubharmonic on $Y$.
\end{lemma}
\end{samepage}

We recall first the classical theorem of Grauert--Remmert
\cite{GR56}: if $X$ is a normal complex space, then a
plurisubharmonic function on $X_{\reg}$ extends uniquely to a
plurisubharmonic function on $X$. In particular, such a
function is locally bounded from above near the singular locus. The next
proposition retains the local upper-boundedness conclusion while replacing
normality by the weaker hypothesis $\dim_{\bC}X_{\sing}\le n-2$.

\begin{proposition}\label{prop:psh-upper}
Let $X$ be a reduced complex space of pure dimension $n$ and assume
\[
 \dim_{\bC}X_{\sing}\le n-2.
\]
Then every $u\in\PSH(X_{\reg})$ is locally bounded from above along
$X_{\sing}$. More precisely, for every $z_0\in X$ there are a neighborhood
$U$ of $z_0$ and $C_U<+\infty$ such that
\[
 u\le C_U\qquad\text{on }U\cap X_{\reg}.
\]
\end{proposition}

\begin{proof}
Fix $z_0\in X_{\sing}$. By the local parametrization theorem, there are a
neighborhood $U\subset X$ of $z_0$, a domain
$\Omega\subset\bC^n$, and a finite proper surjective holomorphic map
\[
 f:U\longrightarrow\Omega.
\]
Put $E=f(U\cap X_{\sing})$. By Remmert's proper mapping theorem \cite{Remmert56}, $E$ is
analytic; since $f$ is finite,
\[
 \dim_{\bC}E=\dim_{\bC}(U\cap X_{\sing})\le n-2,
\]
and hence $\mathcal H^{2n-2}(E)=0$.

Set $D=\Omega\setminus E$. The map $f:f^{-1}(D)\to D$ is a finite proper
surjective holomorphic map between complex manifolds. Define
\begin{equation}\label{eq:umax-local}
 u_{\max}(z):=\max\{u(\zeta):\zeta\in f^{-1}(z)\},\qquad z\in D.
\end{equation}
By Lemma~\ref{lem:maxpush}, $u_{\max}\in\PSH(D)$.
Shiffman's removable-singularity theorem
\cite{Shiffman} applies because
$\mathcal H^{2n-2}(E)=0$, and yields an extension
$\widetilde u_{\max}\in\PSH(\Omega)$.

Here the codimension assumption is used before any upper bound near the
singular locus is known. Passing to the maximum over the finite fibers
reduces the problem to removal of a set of zero
$\mathcal H^{2n-2}$-measure in a Euclidean domain. Once this extension is
obtained, its local upper bounds control every branch of the finite map.

Choose $D_0\Subset\Omega$ containing $f(z_0)$. Then
$C:=\max\{0,\sup_{\overline{D_0}}\widetilde u_{\max}\}<+\infty$. For
$\zeta\in f^{-1}(D_0)\cap X_{\reg}$ with $f(\zeta)\notin E$,
\[
 u(\zeta)\le u_{\max}(f(\zeta))\le C.
\]
At points of $X_{\reg}\cap f^{-1}(E)$, the same estimate follows from the
upper-limit characterization of $u$ across the proper analytic subset
$f^{-1}(E)\cap X_{\reg}$. Thus $u\le C$ on a
neighborhood of $z_0$ in $X_{\reg}$.
\end{proof}

Together with Demailly's extension theorem, the preceding bound gives the
following consequence.

\begin{proposition}\label{prop:psh-extension}
Under the hypotheses of Proposition~\ref{prop:psh-upper}, assume in addition
that $X$ is locally irreducible. Then every $u\in\PSH(X_{\reg})$ admits the
canonical plurisubharmonic extension
\[
 u^*(z)=\limsup_{X_{\reg}\ni\zeta\to z}u(\zeta),\qquad z\in X,
\]
to all of $X$.
\end{proposition}
\begin{proof}
By Proposition~\ref{prop:psh-upper}, $u$ is locally bounded from above along
$X_{\sing}$. The set $X_{\sing}$ is analytic and has empty interior in
$X$. Since $X$ is locally irreducible, Demailly's extension theorem
\cite[Theorem~1.7]{Demailly1985}, applied to $X\setminus X_{\sing}$,
gives a unique plurisubharmonic extension to $X$, with the stated
upper-limit formula.
\end{proof}

Finally, if $D\subset\bC^n$ is open and
$h\in\PSH(D)\cap L^\infty_{\mathrm{loc}}(D)$, then
\begin{equation}\label{eq:BT}
 h\text{ is maximal}\quad\Longleftrightarrow\quad(\ddc h)^n=0
\end{equation}
in the Bedford--Taylor sense \cite{BT76,BT82}. All uses of \eqref{eq:BT} below concern locally bounded functions.

\begin{lemma}\label{lem:spread}
Let $n\geq2$, and suppose that $u=c\log|I|+O(1)$ near $z_0\in\bC^n$, where $I\neq0$ and $c>0$. If $u$ is maximal on a neighborhood of $z_0$, then $\ell(I)\leq n-1$.
\end{lemma}
\begin{proof}
We may assume that $z_0=0$ and that $I$ is proper, since the unit ideal has analytic spread one. Suppose, to the contrary, that $\ell(I)=n$. By \eqref{eq:mcadam} and the definition of an associated prime, there exist $k\geq1$ and a nonzero class $g+\overline{I^k}$ such that
\[
 \Ann_{\cO_{\bC^n,0}}(g+\overline{I^k})=\mm_0.
\]
Since $\mm_0=(z_1,\ldots,z_n)$, this gives
\begin{equation}\label{eq:witness}
 g\notin\overline{I^k},\qquad z_jg\in\overline{I^k},\quad 1\leq j\leq n.
\end{equation}
It follows from \eqref{eq:closure}, after choosing representatives on a sufficiently small ball, that
\begin{equation}\label{eq:z-witness}
 |z|\,|g(z)|\leq A|I(z)|^k
\end{equation}
for some $A>0$. Indeed, apply \eqref{eq:closure} to each $z_jg$ and add the squares of the resulting inequalities. The norm of the tuple of all ordered products of $k$ generators of $I$ is $|I|^k$.

Thus $g$ is controlled by $|I|^k$ on every annulus about the origin, although no such estimate can hold on a full neighborhood of the origin: otherwise \eqref{eq:closure} would contradict $g\notin\overline{I^k}$. The logarithm of $g$ will therefore give a competitor which lies below $u$ near the boundary of a small ball but cannot lie below $u$ throughout that ball. The constants below make this comparison valid for the given bounded remainder.

Choose $r>0$ such that $u$ is maximal on a neighborhood $D$ of $\overline{B(0,r)}$ and
\[
 c\log|I|-M\leq u\leq c\log|I|+M
 \quad\text{on }B(0,r).
\]
By \eqref{eq:z-witness}, $|g|\leq(2A/r)|I|^k$ on $r/2\leq|z|<r$. Set
\[
 a=M+\frac{c}{k}\log\frac{2A}{r},
 \qquad v=\frac{c}{k}\log|g|-a.
\]
On this annulus, the choice of $a$ gives $v\leq c\log|I|-M\leq u$, including zeros of $I$, where \eqref{eq:z-witness} forces $g$ to vanish. The function
\[
 \widetilde u=\begin{cases}
 \max\{u,v\},&\text{on }B(0,r),\\
 u,&\text{on }D\setminus\overline{B(0,r/2)},
 \end{cases}
\]
is therefore plurisubharmonic on $D$ and agrees with $u$ outside $\overline{B(0,r/2)}$. Maximality gives $\widetilde u=u$, so $v\leq u$ on $B(0,r)$. Using the upper bound for $u$, we obtain
\[
 |g|\leq\exp\!\left(\frac{k(a+M)}c\right)|I|^k
 =\frac{2A}{r}\exp\!\left(\frac{2kM}{c}\right)|I|^k
 \quad\text{on }B(0,r)\setminus V(I).
\]
By continuity the inequality holds on the whole ball. The criterion \eqref{eq:closure} implies $g\in\overline{I^k}$, contradicting \eqref{eq:witness}.
\end{proof}

In particular, the pole set of a locally maximal function with locally analytic singularities has no isolated point. Otherwise the ideal defining the singularity at such a point $z$ would be $\mm_z$-primary, hence would have analytic spread $n$, contradicting Lemma~\ref{lem:spread}.

The next lemma constructs a coherent ideal from finitely many local expressions for the singularities of $u$. The coefficient may differ on open sets whose intersections do not meet $E$. Directly gluing the defining ideals would be too restrictive, since the same logarithmic singularity can be represented by different powers of an ideal. The proof first makes the coefficients compatible by taking powers and then uses integral closure to turn comparability of norms into equality of ideal sheaves. Only finitely many presentations are needed, since the final comparison takes place near a fixed compact set.

\begin{lemma}\label{lem:glue}
Let $D_0=\bigcup_{j=1}^N D_j$, where each $D_j\subset\bC^n$ is open. Suppose that $u\in\PSH(D_0)$ is not identically $-\infty$ on any connected component of $D_0$, and set $E=\{u=-\infty\}$. Assume that on each $D_j$,
\[
 u=c_j\log|\cI_j|+O(1),\qquad c_j>0,
\]
with $\cI_j$ a coherent ideal whose stalk is nonzero at every point of $D_j$. There are a partition $\mathcal A$ of $\{1,\ldots,N\}$, coherent ideals $\cJ_A\subset\cO_{D_0}$ with pairwise disjoint zero sets, and positive constants $c_A$, indexed by $A\in\mathcal A$, such that, for $\cJ=\prod_{A\in\mathcal A}\cJ_A$,
\begin{equation}\label{eq:commonmodel}
 u=c_A\log|\cJ|+O(1)\quad\text{on }D_j,\quad j\in A.
\end{equation}
Moreover, there are integers $N_j\geq1$ such that, for every $j\in A$,
\[
 c_j=N_jc_A,\qquad
 \cJ|_{D_j}=\cJ_A|_{D_j}=\overline{\cI_j^{N_j}}.
\]
On $D_j$, all $\cJ_{A'}$ with $A'\neq A$ are unit ideals.
\end{lemma}
\begin{proof}
Partition the indices into equivalence classes generated by the relation $D_j\cap D_k\cap E\neq\varnothing$. If $D_j\cap D_k\cap E\neq\varnothing$, choose $z\in D_j\cap D_k\cap E$ and write
\[
 m_j=\ord_z(\cI_j),\qquad m_k=\ord_z(\cI_k).
\]
These are positive integers. Choose a generator of minimal order from each ideal and a vector $a$ on which both lowest nonzero homogeneous terms are nonzero. On the line $z+ta$, the two ideal norms are comparable to $|t|^{m_j}$ and $|t|^{m_k}$. Comparing the two presentations of $u$ yields
\[
 (c_jm_j-c_km_k)\log|t|=O(1),
 \qquad \frac{c_j}{c_k}=\frac{m_k}{m_j}\in\mathbb Q_{>0}.
\]
It follows that $c_j/c_k$ is rational whenever $j$ and $k$ belong to the same equivalence class. Clearing finitely many denominators, choose $c_A>0$ and integers $N_j\geq1$ such that $c_j=N_jc_A$ for all $j\in A$.

The constants $c_A$ themselves need not be rational. Only the ratios of
coefficients within one class are constrained, because an overlap meeting
$E$ supplies a common logarithmic pole along which they can be compared.
On an overlap disjoint from $E$, both ideal norms are locally bounded
away from zero, so the bounded-remainder relation imposes no such
constraint. This is why a separate coefficient is retained for each class.

We claim that the integral closures agree on every overlap of two indices in the same class. If $D_j\cap D_k$ does not meet $E$, both ideals are the unit ideal there and there is nothing to prove. Otherwise, on the overlap the two presentations and $c_j=N_jc_A$, $c_k=N_kc_A$ give
\[
 N_j\log|\cI_j|-N_k\log|\cI_k|=O(1)
\]
on the complement of $E$. Hence, after shrinking around an arbitrary point of the overlap, there is $C>0$ such that
\[
 |\cI_j|^{N_j}\le C|\cI_k|^{N_k},\qquad
 |\cI_k|^{N_k}\le C|\cI_j|^{N_j}
\]
off $E$. The norms are continuous and vanish on $E$, so these inequalities extend across $E$. Applying the analytic criterion for integral dependence \eqref{eq:closure} in the two directions yields
\begin{equation}\label{eq:overlap}
 \overline{\cI_j^{N_j}}\big|_{D_j\cap D_k}
 =\overline{\cI_k^{N_k}}\big|_{D_j\cap D_k}.
\end{equation}
For a fixed equivalence class $A$, prescribe
\[
 \cJ_A|_{D_j}=\begin{cases}
 \overline{\cI_j^{N_j}},&j\in A,\\
 \cO_{D_j},&j\notin A.
 \end{cases}
\]
An overlap between sets whose indices belong to different equivalence classes contains no pole. Both prescribed ideals are then the unit ideal. Together with \eqref{eq:overlap}, this gives compatible local ideals, which glue to a coherent ideal $\cJ_A$.

On $D_j$, $j\in A$, every other factor is the unit ideal, so $\cJ=\overline{\cI_j^{N_j}}$. Let $g$ be a finite tuple of local generators of this integral closure. Applying \eqref{eq:closure} to each component gives $|g|\leq C|\cI_j|^{N_j}$. The reverse estimate follows from $\cI_j^{N_j}\subset\overline{\cI_j^{N_j}}$ by expressing generators of $\cI_j^{N_j}$ in terms of the components of $g$. Hence
\[
 \log|\cJ|=N_j\log|\cI_j|+O(1),
\]
which proves \eqref{eq:commonmodel}. Since at every point at most one factor is nonunit, the zero sets of the $\cJ_A$ are disjoint.
\end{proof}

The same argument shows that analytic spread is independent of the ideal used to write $u=c\log|I|+O(1)$. Indeed, if
\[
 u=c\log|I|+O(1)=d\log|J|+O(1)
\]
near a point of $E$, comparison along a generic line gives $c/d\in\mathbb Q_{>0}$. Choose positive integers $p,q$ such that $c/d=p/q$. Then $p\log|I|-q\log|J|=O(1)$, so \eqref{eq:closure} gives $\overline{I^p}=\overline{J^q}$. Equation~\eqref{eq:spread-invariance} now yields $\ell(I)=\ell(J)$. Away from $E$, every ideal in such an expression is the unit ideal.

We now record the properties of the normalized blow-up needed in the proof.
\begin{lemma}\label{lem:fiber}
Let $\cJ\neq0$ be a coherent ideal on a domain $D\subset\bC^n$, and let $\pi:Y\to D$ be its normalized blow-up. Then
\[
 \dim\pi^{-1}(z)=\ell(\cJ_z)-1.
\]
If $\ell(\cJ_z)\leq n-1$ at every point, no prime component of the divisor of $\cJ\cO_Y$ has a point as its image.
\end{lemma}
\begin{proof}
The fiber of the ordinary blow-up $\pi_0:X=\Bl_{\cJ}D\to D$ over $z$ is the projective analytic space associated to the graded algebra
\[
 \mathcal F(\cJ_z)=\bigoplus_{k=0}^{\infty}\cJ_z^k/\mm_z\cJ_z^k.
\]
This follows from the relative $\Proj$ construction after base change to the residue field. Its dimension is $\dim\mathcal F(\cJ_z)-1$. Normalization $\eta:Y\to X$ is finite and surjective and remains so over each fiber; it therefore preserves fiber dimension. In the unit-ideal case, the fiber is one point and the formula follows from $\mathcal F(\cO_{D,z})=\bC[t]$.

Both blow-ups have pure dimension $n$. If a prime divisor $H$ maps to $z$, then
\[
 n-1=\dim H\leq\dim\pi^{-1}(z)=\ell(\cJ_z)-1,
\]
contradicting the stated bound.
\end{proof}

The preceding conclusion concerns the normalized blow-up. A further resolution may introduce divisors whose images are points. Such a divisor could lie entirely over the compact set on which a competitor differs from $u$, leaving no open part where the two pullbacks are known to agree. We therefore work on $Y$ and use the extension theorem for plurisubharmonic functions on normal spaces. This preserves the geometric consequence of the spread bound while still allowing subtraction of logarithmic potentials across the singular locus of $Y$.

\begin{lemma}\label{lem:center}
Let $K\Subset D$ be compact, and let $H\subset Y$ be a closed irreducible analytic hypersurface whose image under $\pi$ has positive dimension. Then $H\cap\pi^{-1}(D\setminus K)$ is a nonempty relatively open subset of $H$.
\end{lemma}
\begin{proof}
The restriction of $\pi$ to $H$ is proper. By Remmert's proper mapping theorem \cite{Remmert56}, $Z=\pi(H)$ is a closed analytic subset of $D$. If $Z\subset K$, then $Z$ is compact. Every irreducible component of a compact analytic subset of a domain in $\bC^n$ is a point: the maximum principle on compact complex spaces makes every coordinate function constant. This contradicts $\dim Z>0$. Openness follows from openness of $D\setminus K$.
\end{proof}

For the real divisors used below, a logarithmic potential is understood as follows. If locally $\Delta=\sum_{l=1}^{r} a_l B_l$, where $r$ is finite, $a_l>0$, and each $B_l$ is an effective Cartier divisor with holomorphic local equation $g_l$, set
\[
 \varphi=\sum_{l=1}^{r} a_l\log|g_l|.
\]
Then $\ddc\varphi=[\Delta]$ on the regular locus. Changing the local equations changes $\varphi$ by a pluriharmonic function.

\begin{lemma}\label{lem:cartier}
For the ideals in Lemma~\ref{lem:glue}, restricted to a domain $D\subset D_0$, let $\pi:Y\to D$ be the normalized blow-up of $\cJ$. Enumerate the classes in the partition as $A_1,\ldots,A_s$. Each $\cJ_A\cO_Y$ is invertible, say $\cO_Y(-B_A)$, and
\[
 \Delta=\sum_{k=1}^{s} c_{A_k}B_{A_k}
\]
is an effective real Cartier divisor. Locally on $Y$, the function $u\circ\pi$ differs from a logarithmic potential of $\Delta$ by a bounded function.
\end{lemma}
\begin{proof}
Near each base point at most one $\cJ_A$ is nonunit, and there $\cJ=\cJ_A$. Since $\cJ\cO_Y$ is invertible, each $\cJ_A\cO_Y$ is invertible on that inverse image. Where $\cJ_A$ is the unit ideal, its pullback is also the unit ideal. These neighborhoods cover $Y$.

Near a point of $\pi^{-1}(D_k)$ with $k\in A$, choose generators $f_j$ of $\cJ$ and a local generator $h$ of its pullback. Write
\[
 f_j\circ\pi=h g_j,\qquad (g_1,\ldots,g_m)=\cO_Y.
\]
On a smaller neighborhood, $|g|$ is bounded above and away from zero. Thus
\[
 \log|\cJ|\circ\pi=\log|h|+O(1),\qquad
 u\circ\pi=c_A\log|h|+O(1).
\]
The function $c_A\log|h|$ is the required local potential. Off the total transform of $E$, the divisor is zero and $u\circ\pi$ is locally bounded. The divisors $B_A$ are Cartier, although their individual prime components need not be Cartier; the logarithmic potentials are defined using local equations for the $B_A$.

Accordingly, the expression $\Delta=\sum_H d_HH$ records the
coefficients to be compared by Lelong numbers, whereas the expression
$\Delta=\sum_A c_AB_A$ provides the local potentials to be subtracted.
These two descriptions serve different purposes. No choice of a
holomorphic defining equation for an individual prime component at a
singular point of $Y$ is required.
\end{proof}

\begin{lemma}\label{lem:subtract}
Let $Y$ be a normal complex space, and let $\Delta$ be locally a finite positive real linear combination of effective Cartier divisors. Denote by $\mathcal H$ the locally finite family of prime components of $\Delta$, and write $\Delta=\sum_{H\in\mathcal H} d_H H$. On an open set $W$, choose a logarithmic potential $\varphi$ for $\Delta$ as above. Suppose that $\psi\in\PSH(Y)$ is not identically $-\infty$ on any irreducible component of $Y$, and
\[
 \nu_H(\ddc\psi)\geq d_H
\]
for every component of $\Delta$, with the coefficients computed on $Y_\reg$. Then $\psi-\varphi$, initially defined off $|\Delta|$, has a plurisubharmonic extension to $W$. If this difference is locally bounded on both sides off $|\Delta|$, its extension is locally bounded.
\end{lemma}
\begin{proof}
The difference is initially considered only off $|\Delta|$, where
$\varphi$ is finite. At points of $|\Delta|$, both terms may be
$-\infty$, so their pointwise difference is not defined. The argument
first interprets the subtraction as a difference of locally integrable
potentials on the regular locus, and then recovers its canonical
plurisubharmonic representative.

On $W\cap Y_\reg$, the functions $\psi$ and $\varphi$ are locally integrable. By Siu's decomposition \cite{Siu} and the hypothesis on the generic Lelong numbers,
\[
 \ddc(\psi-\varphi)=\ddc\psi-[\Delta]\geq0.
\]
Thus $\psi-\varphi$ has a plurisubharmonic representative $h$ on
$W\cap Y_\reg$. Off $|\Delta|$, the function $\varphi$ is pluriharmonic, so
this representative agrees with the original difference. A normal complex
space is locally pure-dimensional, locally irreducible, and regular in
codimension one. Proposition~\ref{prop:psh-upper}, applied locally to
$h\in\PSH(W\cap Y_\reg)$, shows first that $h$ is locally bounded from
above along $W\cap Y_\sing$. Since a normal complex space is locally
irreducible, Proposition~\ref{prop:psh-extension} then gives the unique
plurisubharmonic extension of $h$ to $W$, still denoted by $h$, with
\[
 h(z)=\limsup_{\substack{\zeta\to z\\\zeta\in W\cap Y_\reg}}h(\zeta),
 \qquad z\in W\cap Y_\sing.
\]

Suppose that $-a\leq\psi-\varphi\leq a$ off $|\Delta|$ on a neighborhood. The upper-limit characterization of a plurisubharmonic extension gives the same bounds for $h$, first across $|\Delta|\cap Y_\reg$ and then across $Y_\sing$. Hence $h$ is locally bounded.

For local potentials $\varphi_j$ on open sets $W_j$, the corresponding extensions satisfy
\begin{equation}\label{eq:transitions}
 h_j=h_k+\varphi_k-\varphi_j
 \quad\text{on }W_j\cap W_k.
\end{equation}
Here $\varphi_k-\varphi_j$, initially defined off $|\Delta|$, extends pluriharmonically to the overlap, since local equations of a Cartier divisor differ by holomorphic units. Equality \eqref{eq:transitions} follows off $|\Delta|$ and then everywhere by uniqueness of the plurisubharmonic extensions.
\end{proof}

\section{Proof of the main theorem}

We first establish the comparison of singularities in
Proposition~\ref{prop:intro-rigidity}. This part of the argument uses the
spread bound and the equality of the two functions outside a compact set;
it does not use maximality away from $E$. That analytic hypothesis enters
only in Lemma~\ref{lem:comparison}, after the difference of the two
functions has been bounded near the pole set.

\begin{proof}[Proof of Proposition~\ref{prop:intro-rigidity}]
Let $K\Subset\Omega$ be compact and $v=u$ on $\Omega\setminus K$. Fix $z\in E$. Choose a bounded connected domain $D\Subset\Omega$ containing $K\cup\{z\}$. Cover $\overline D$ by finitely many open sets on which $u$ has an expression of the form \eqref{eq:model-intro} and apply Lemma~\ref{lem:glue} on their union $D_0$. Restrict the resulting ideals to $D$ and form the normalized blow-up $\pi:Y\to D$ of $\cJ$.

On $D_j$ with $j\in A$, Lemma~\ref{lem:glue} gives $\cJ=\overline{\cI_j^{N_j}}$ and $u=c_A\log|\cJ|+O(1)$. By \eqref{eq:spread-invariance} and the hypothesis of the proposition,
\[
 \ell(\cJ_{z'})=\ell((\cI_j)_{z'})<n\quad\text{for every }z'\in D_j\cap D.
\]
By Lemma~\ref{lem:fiber}, the image of every prime component of $B$ has positive dimension, where $\cJ\cO_Y=\cO_Y(-B)$. The local expression for $u$ and the equality of zero sets under integral closure give
\[
 V(\cJ)=E\cap D,\qquad |B|=\pi^{-1}(E\cap D).
\]

Let $\Delta=\sum_{k=1}^{s}c_{A_k}B_{A_k}=\sum_{H\in\mathcal H} d_H H$ be the real Cartier divisor from Lemma~\ref{lem:cartier}, where $\mathcal H$ denotes its locally finite family of prime components. At a general smooth point of a component $H$, choose a coordinate $t$ defining $H$ and avoid the other components. A local equation for $B_A$ is then a holomorphic unit times $t^{b_H}$. Therefore
\[
 u\circ\pi=c_Ab_H\log|t|+O(1),\qquad
 \nu_H(\ddc(u\circ\pi))=c_Ab_H=d_H.
\]
Normality gives $\operatorname{codim}_Y Y_\sing\geq2$, so $H$ is not contained in $Y_\sing$. The manifold $H_\reg\setminus Y_\sing$ is connected: $H_\reg$ is connected by irreducibility, and removing a proper analytic subset preserves connectedness. It is dense in $H$, and the generic divisorial Lelong numbers are computed there.

By Lemma~\ref{lem:center}, choose a nonempty relatively open subset $U_H\subset H_\reg\setminus Y_\sing$ over $D\setminus K$. The pullbacks $u\circ\pi$ and $v\circ\pi$ agree on a neighborhood of $U_H$, so their Lelong numbers agree there. By Siu's theorem \cite{Siu}, each current attains its generic value outside a countable union of proper analytic subsets of $H_\reg\setminus Y_\sing$. These exceptional subsets cannot cover $U_H$. Choosing a point outside both exceptional sets gives equality of the generic divisorial coefficients, hence
\begin{equation}\label{eq:equal-lelong}
 \nu_H(\ddc(v\circ\pi))=d_H\quad\text{for every }H\subset B.
\end{equation}

The hypothesis $v\geq u$ alone would give only the opposite comparison
of singularities, namely
$\nu_H(\ddc(v\circ\pi))\leq\nu_H(\ddc(u\circ\pi))$.
Equality outside $K$ is what rules out a loss of the divisorial
coefficient. Lemma~\ref{lem:center} makes this equality accessible on
every component, including those whose images meet $K$.

Since $D$ is a domain and $\pi$ is a modification, $Y$ is irreducible. Both $u\circ\pi$ and $v\circ\pi$ are plurisubharmonic on $Y$. They are finite on $\pi^{-1}(D\setminus E)$, because $u$ is locally bounded on $D\setminus E$ and $v\geq u$. Hence neither pullback is identically $-\infty$ on $Y$, as required in Lemma~\ref{lem:subtract}. Applying that lemma to both pullbacks, with local potentials $\varphi_j$ of $\Delta$ on open sets $W_j$, we obtain
\[
 h_j\in\PSH(W_j)\cap L^\infty_{\mathrm{loc}}(W_j),\qquad
 \psi_j\in\PSH(W_j),
\]
where off $|B|$,
\begin{equation}\label{eq:residual-difference}
 h_j=u\circ\pi-\varphi_j,\quad
 \psi_j=v\circ\pi-\varphi_j,\quad
 \psi_j-h_j=(v-u)\circ\pi.
\end{equation}
By \eqref{eq:transitions}, the differences $\psi_j-h_j$ agree on overlaps and coincide with $(v-u)\circ\pi$ off $|B|$.

The boundedness of $h_j$ uses the two-sided analytic model for $u$.
For $\psi_j$, plurisubharmonicity already supplies the local upper bound
needed below. Thus equality of divisorial coefficients is combined with
the bounded residual part of $u$ to control the full difference; equality
of those coefficients alone is not being used as a bounded-difference
criterion for arbitrary plurisubharmonic functions.

All these constructions are local on $Y$. The divisor may have many
prime components globally, but its family of components is locally
finite, so every logarithmic potential involves only finitely many
terms. We likewise do not require a single residual potential on all of
$Y$: changes of local equations add the same pluriharmonic function to
both residual potentials and cancel in their difference. The remaining
step is to turn these local bounds into one bound over a neighborhood
of the chosen base point.

The fiber $F_z=\pi^{-1}(z)$ is compact. Cover it by finitely many relatively compact open sets $U_k\Subset W_{j(k)}$, $1\leq k\leq r$. Local upper bounds for $\psi_{j(k)}$ and local lower bounds for $h_{j(k)}$ give a finite constant $C_z$ such that
\[
 (v-u)\circ\pi\leq C_z\quad\text{on }N\setminus|B|,
 \qquad N=\bigcup_{k=1}^{r}U_k.
\]
For example, one can take the maximum of zero and the finitely many
numbers
$\sup_{\overline U_k}\psi_{j(k)}-\inf_{\overline U_k}h_{j(k)}$.
Covering the entire fiber is essential here: an estimate near a single
point of $F_z$ need not control all points lying over nearby base points.
Properness implies that $\pi(Y\setminus N)$ is closed in $D$. Since this image does not contain $z$, the set
\[
 D_z=D\setminus\pi(Y\setminus N)
\]
is an open neighborhood of $z$, and $\pi^{-1}(D_z)\subset N$. The map $\pi$ is an isomorphism over $D\setminus E$, so the last bound descends to $v-u\leq C_z$ on $D_z\setminus E$. The lower bound follows from $v\geq u$.
\end{proof}

The following lemma is due to Rashkovskii--Sigurdsson \cite[Lemma~4.1]{RS}; we include its proof for completeness.

\begin{lemma}\label{lem:comparison}
Let $D\subset\bC^n$ be a bounded domain. Suppose $u,v\in\PSH(D)$ are bounded above, $E=\{u=-\infty\}$ is a proper closed subset, $u\in L^\infty_{\mathrm{loc}}(D\setminus E)$, and
\[
 (\ddc u)^n=0\qquad\text{on }D\setminus E
\]
in the Bedford--Taylor sense. Assume $v=u$ outside a compact set $K\Subset D$, and $v-u$ is bounded above locally near every point of $E$, off $E$. Then $v\leq u$ on $D$.
\end{lemma}
\begin{proof}
Since $E=\{u=-\infty\}$ is a proper subset of the domain $D$, the function $u$ is not identically $-\infty$, and $E$ is pluripolar.

The set $D\setminus E$ is open in $\bC^n$, and $u$ is locally bounded there. The hypothesis $(\ddc u)^n=0$ and the Bedford--Taylor characterization~\eqref{eq:BT} imply that $u$ is maximal on $D\setminus E$.

Subtracting a common constant, we may assume that $u,v<0$. Fix $\varepsilon>0$ and put $v_\varepsilon=(1+\varepsilon)v$. We claim that
\begin{equation}\label{eq:compact-comparison}
 \{v_\varepsilon>u\}\Subset D\setminus E.
\end{equation}
Multiplication by $1+\varepsilon$ preserves plurisubharmonicity and makes
the negative competitor more singular near $E$. The local boundedness
of $v-u$ ensures that this change absorbs its possible positive excess
there. The remaining comparison set can then be treated entirely where
$u$ is locally bounded.
First, $v_\varepsilon\leq u$ on $D\setminus K$. Let $z\in K\cap E$. By assumption, there are a neighborhood $D_z$ and a constant $a_z\geq0$ such that $v\leq u+a_z$ on $D_z\setminus E$. Consequently
\[
 v_\varepsilon-u\leq\varepsilon u+(1+\varepsilon)a_z
 \quad\text{on }D_z\setminus E.
\]
Since $u(z)=-\infty$, upper semicontinuity allows us to shrink $D_z$ so that the right-hand side is nonpositive. Moreover, the inequality $v\leq u+a_z$ off $E$ implies $v=-\infty$ on $D_z\cap E$, by the upper-limit characterization of plurisubharmonic functions across a pluripolar set. Thus $v_\varepsilon\leq u$ throughout $D_z$.

Choose finitely many such neighborhoods covering $K\cap E$, and denote their union by $D_0$. Then
\[
 \{v_\varepsilon>u\}\subset K\setminus D_0\Subset D\setminus E,
\]
which proves \eqref{eq:compact-comparison}. Maximality of $u$ on $D\setminus E$ gives $v_\varepsilon\leq u$ there. The same inequality holds on $E$, where $u=v=-\infty$. Letting $\varepsilon\downarrow0$, we conclude that $v\leq u$.
\end{proof}

\begin{proof}[Proof of Corollary~\ref{cor:criterion}]
Maximality passes to open subsets, so \textup{(i)} implies \textup{(ii)}. Lemma~\ref{lem:spread} gives \textup{(ii)}$\Rightarrow$\textup{(iii)}.

Assume \textup{(iii)}. Away from $E$, each defining ideal is the unit ideal and has analytic spread one. Thus the spread bound holds on every $D_j$. Let $\psi\in\PSH(\Omega)$ satisfy $\{\psi>u\}\Subset\Omega$, set $v=\max(u,\psi)$, and let $K=\overline{\{\psi>u\}}$. If $K$ is empty there is nothing to prove. Otherwise choose a bounded domain $K\Subset D\Subset\Omega$. Proposition~\ref{prop:intro-rigidity} bounds $v-u$ above locally near $E\cap D$, off $E$.

Since $u$ is locally bounded on $D\setminus E$ and locally maximal there, \eqref{eq:BT} gives $(\ddc u)^n=0$ on $D\setminus E$. Both $u$ and $v$ are bounded above on $D$, since they are plurisubharmonic on a neighborhood of $\overline D$. Moreover, $v=u$ outside $K$. Lemma~\ref{lem:comparison} therefore yields $v\leq u$ on $D$, and hence $\psi\leq u$ on $\Omega$. This proves \textup{(i)}.

Finally, the equivalence between local maximality and the vanishing of $(\ddc u)^n$ on $\Omega\setminus E$ follows from \eqref{eq:BT}, since $u$ is locally bounded there.
\end{proof}

\begin{proof}[Proof of Theorem~\ref{thm:main}]
The local holomorphic tuples in \eqref{eq:model-intro} generate coherent ideals with nonzero stalks. The conclusion is therefore the implication \textup{(ii)}$\Rightarrow$\textup{(i)} in Corollary~\ref{cor:criterion}.
\end{proof}

\section*{Acknowledgements}
This work was completed while the first author was visiting the Department of Mathematics at the National University of Singapore as a Visiting Senior Research Fellow under the Singapore Academies Southeast Asia Fellowship (SASEAF) Programme, within the project ``Math, Sobolev Space'' (Grant No. E-146-00-0039-01). The first author would like to express his sincere gratitude to the Department of Mathematics at NUS for its warm hospitality and excellent working environment. He also gratefully acknowledges the Singapore National Academy of Science (SNAS) and the National Research Foundation (NRF), Singapore, for their financial support through the SASEAF Programme.


\begin{thebibliography}{ACCP09}
\bibitem[ACCP09]{ACCP09} P.~{\AA}hag, U.~Cegrell, R.~Czy\.z and H.~H.~Pham, Monge--Amp\`ere measures on pluripolar sets, \emph{J. Math. Pures Appl.} \textbf{92}, no.~6, 613--627 (2009).
\bibitem[AC07]{AC07} P.~{\AA}hag and R.~Czy\.z, On the Cegrell classes, \emph{Math. Z.} \textbf{256}, no.~2, 243--264 (2007).
\bibitem[ABW19]{ABW} M.~Andersson, Z.~B\l ocki and E.~Wulcan, On a Monge--Amp\`ere operator for plurisubharmonic functions with analytic singularities, \emph{Indiana Univ. Math. J.} \textbf{68}, no.~4, 1217--1231 (2019).
\bibitem[BT76]{BT76} E.~Bedford and B.~A.~Taylor, The Dirichlet problem for a complex Monge--Amp\`ere equation, \emph{Invent. Math.} \textbf{37}, 1--44 (1976).
\bibitem[BT82]{BT82} E.~Bedford and B.~A.~Taylor, A new capacity for plurisubharmonic functions, \emph{Acta Math.} \textbf{149}, 1--40 (1982).
\bibitem[B04]{B04} Z.~B\l ocki, On the definition of the Monge--Amp\`ere operator in $\bC^2$, \emph{Math. Ann.} \textbf{328}, no.~3, 415--423 (2004).
\bibitem[B06]{B06} Z.~B\l ocki, The domain of definition of the complex Monge--Amp\`ere operator, \emph{Amer. J. Math.} \textbf{128}, no.~2, 519--530 (2006).
\bibitem[B09]{B09} Z.~B\l ocki, A note on maximal plurisubharmonic functions, \emph{Uzbek Math. J.} 2009, no.~1, 28--32 (2009).
\bibitem[C98]{C98} U.~Cegrell, Pluricomplex energy, \emph{Acta Math.} \textbf{180}, no.~2, 187--217 (1998).
\bibitem[C04]{C04} U.~Cegrell, The general definition of the complex Monge--Amp\`ere operator, \emph{Ann. Inst. Fourier (Grenoble)} \textbf{54}, no.~1, 159--179 (2004).
\bibitem[C08]{C08} U.~Cegrell, A general Dirichlet problem for the complex Monge--Amp\`ere operator, \emph{Ann. Polon. Math.} \textbf{94}, no.~2, 131--147 (2008).
\bibitem[C09]{C09} U.~Cegrell, Maximal plurisubharmonic functions, \emph{Uzbek Math. J.} 2009, no.~1, 10--16 (2009).
\bibitem[D85]{Demailly1985} J.-P.~Demailly, Mesures de Monge--Amp\`ere et caract\'erisation g\'eom\'etrique des vari\'et\'es alg\'ebriques affines, \emph{M\'em. Soc. Math. France (N.S.)} \textbf{19}, 1--124 (1985).
\bibitem[D93]{D93} J.-P.~Demailly, Monge--Amp\`ere operators, Lelong numbers and intersection theory, in \emph{Complex Analysis and Geometry}, V.~Ancona and A.~Silva (eds.), Univ. Ser. Math., Plenum Press, New York, 115--193 (1993).
\bibitem[GR56]{GR56} H.~Grauert and R.~Remmert, Plurisubharmonische Funktionen in komplexen R\"aumen, \emph{Math. Z.} \textbf{65}, 175--194 (1956).
\bibitem[LT08]{LT} M.~Lejeune-Jalabert and B.~Teissier, Cl\^oture int\'egrale des id\'eaux et \'equisingularit\'e, with an appendix by J.-J.~Risler, \emph{Ann. Fac. Sci. Toulouse Math.} (6) \textbf{17}, no.~4, 781--859 (2008).
\bibitem[McA80]{McA} S.~McAdam, Asymptotic prime divisors and analytic spreads, \emph{Proc. Amer. Math. Soc.} \textbf{80}, no.~4, 555--559 (1980).
\bibitem[NP09]{KH09} V.~K.~Nguyen and H.~H.~Pham, A comparison principle for the complex Monge--Amp\`ere operator in Cegrell's classes and applications, \emph{Trans. Amer. Math. Soc.} \textbf{361}, no.~10, 5539--5554 (2009).
\bibitem[Oka51]{Oka51} K.~Oka, Sur les fonctions analytiques de plusieurs variables. VIII. Lemme fondamental, \emph{J. Math. Soc. Japan} \textbf{3}, 204--214 and 259--278 (1951).
\bibitem[P06]{H06} H.~H.~Pham, The comparison principle and Dirichlet problem in the class $\mathcal E_p(f)$, $p>0$, \emph{Ann. Polon. Math.} \textbf{88}, no.~3, 247--261 (2006).
\bibitem[R98]{R98} A.~Rashkovskii, Maximal plurisubharmonic functions associated with holomorphic mappings, \emph{Indiana Univ. Math. J.} \textbf{47}, no.~1, 297--309 (1998).
\bibitem[R09]{R09} A.~Rashkovskii, Analyticity and propagation of plurisubharmonic singularities, in \emph{Functional Analysis and Complex Analysis}, Contemporary Mathematics, vol.~481, American Mathematical Society, Providence, RI, 137--143 (2009).
\bibitem[RS05]{RS} A.~Rashkovskii and R.~Sigurdsson, Green functions with singularities along complex spaces, \emph{Internat. J. Math.} \textbf{16}, no.~4, 333--355 (2005).
\bibitem[Rem56]{Remmert56} R.~Remmert, Projektionen analytischer Mengen, \emph{Math. Ann.} \textbf{130}, 410--441 (1956).
\bibitem[S72]{Shiffman} B.~Shiffman, Extension of positive line bundles and meromorphic maps, \emph{Invent. Math.} \textbf{15}, 332--347 (1972).
\bibitem[S74]{Siu} Y.-T.~Siu, Analyticity of sets associated to Lelong numbers and the extension of closed positive currents, \emph{Invent. Math.} \textbf{27}, 53--156 (1974).
\end{thebibliography}
\end{document}